\documentclass[11pt, a4paper,reqno]{amsart}

\usepackage{amssymb}

\usepackage[T1]{fontenc}
\usepackage[utf8]{inputenc}

\usepackage{enumitem}
\setlist[enumerate,1]{label=\rm(\arabic*)}
\setlist[enumerate,2]{label=\rm(\alph*)}
\setlist[enumerate,3]{label=\rm(\roman*)}

\usepackage[PS]{diagrams}
\usepackage{xcolor}
\usepackage{pgf}

\newcommand{\iso}{\cong}
\newcommand{\ie}{\textit{i.e.}}

\newcommand{\resp}{\textit{resp.}}
\newcommand{\viz}{\textit{viz.}}
\newcommand{\ignore}[1]{\relax}
\newcommand{\bN}{\mathbb{N}}
\newcommand{\bZ}{\mathbb{Z}}

\newcommand{\nbd}{\nobreakdash}
\newcommand{\id}{\ensuremath{\mathrm{id}}}
\newcommand{\tensor}{\otimes}

\newcommand{\nix}{\, \hbox{-} \,}

\numberwithin{equation}{section}
\numberwithin{section}{part}   

\newtheorem{theorem}[equation]{Theorem}
\newtheorem{corollary}[equation]{Corollary}
\newtheorem{proposition}[equation]{Proposition}
\newtheorem{lemma}[equation]{Lemma}

\theoremstyle{definition}
\newtheorem{definition}[equation]{Definition}

\newtheorem{claim}[equation]{Claim}

\let\oldtocsubsection=\tocsubsection
\renewcommand{\tocsubsection}[2]{\hspace{3.5em}\(\cdot\)~\oldtocsubsection{#1}{#2}}

\usepackage[hypertexnames=false,colorlinks=false,pdfborderstyle={/S/U/W 1}]{hyperref}

\newcommand{\cC}{\mathcal{C}}
\newcommand{\cF}{\mathcal{F}}
\newcommand{\cK}{\mathcal{K}}
\newcommand{\cO}{\mathcal{O}}

\newcommand{\cS}{\mathcal{S}}
\newcommand{\cW}{\mathcal{W}}
\newcommand{\cX}{\mathcal{X}}
\newcommand{\cY}{\mathcal{Y}}
\newcommand{\cZ}{\mathcal{Z}}
\newcommand{\inv}{^{-1}}

\newcommand{\sdot}{\cS_{\bullet}}
\newcommand{\R}{R}
\newcommand{\Rpn}[1]{R_{\geq {#1}}}
\newcommand{\Rp}{R_{\geq 0}}
\newcommand{\Rnn}[1]{R_{\leq {#1}}}
\newcommand{\Rn}{R_{\leq 0}}

\newcommand{\bP}{\mathbb{P}}
\newcommand{\pp}{\bP^{1}}
\newcommand{\qco}{\mathfrak{QCoh}(\pp)}
\newcommand{\ch}{\mathrm{Ch}^\flat\,}

\newcommand{\vb}{\mathrm{Vect}(\pp)}
\newcommand{\chm}{\ch \mathrm{\bf M}(R_0)}
\newcommand{\chp}{\ch \mathrm{\bf P}(R_0)}

\newcommand{\coker}{\mathrm{coker}}

\newcommand{\invlim}{\lim\limits_{\leftarrow}}

\newcommand{\pou}[1]{partition of unity of type $(#1, -#1)$}
\newcommand{\apou}[1]{partition of unity of type $(-#1, #1)$}
\newcommand{\hh}[2]{H^{#1} (\pp, #2)}
\newcommand{\Hh}[2]{H^{#1} \big(\pp, #2\big)}
\newcommand{\tw}[3]{#1(#2,\,#3)}
\newcommand{\Tw}[3]{#1\big(#2,\,#3\big)}
\newcommand{\OO}[2]{\tw{\cO} {#1} {#2}}

\renewcommand{\rho}{\varrho}

\renewcommand{\theta}{\vartheta}

\begin{document}

\title[The projective line associated with a strongly $\bZ$-graded ring]%
{The algebraic $K$-theory of the projective line\\associated with a strongly $\bZ$-graded ring}

\date{\today}

\author{Thomas H\"uttemann}
\author{Tasha Montgomery}

\address{Thomas H\"uttemann and Tasha Montgomery\\ Queen's University
  Belfast\\ School of Mathematics and Physics\\ Mathematical Sciences
  Research Centre\\ Belfast BT7~1NN\\ UK}

\email{t.huettemann@qub.ac.uk}

\urladdr{https://t-huettemann.github.io/}

\urladdr{https://pure.qub.ac.uk/en/persons/thomas-huettemann}


\maketitle

\begin{abstract}
  A \textsc{Laurent} polynomial ring $R_{0}[t,t\inv]$ with
  coefficients in a unital ring determines a category of
  quasi-coherent sheaves on the projective line over~$R_{0}$; its
  $K$-theory is known to split into a direct sum of two copies of the
  $K$-theory of~$R_{0}$. In this paper, the result is generalised to
  the case of an arbitrary strongly $\bZ$-graded ring~$\R$ in place of
  the \textsc{Laurent} polynomial ring. The projective line associated
  with~$\R$ is indirectly defined by specifying the corresponding
  category of quasi-coherent sheaves. Notions from algebraic geometry
  like sheaf cohomology and twisting sheaves are transferred to the new
  setting, and the $K$-theoretical splitting is established.
\end{abstract}

\tableofcontents

\part*{Introduction}

\subsection*{The projective line and its algebraic $K$-theory}

The projective line $\pp = \pp_{\R_{0}}$ over a commutative unital
ring~$R_0$ is, by definition, the scheme $\mathrm{Proj}\,R_0[x,y]$; it
is covered by 
open subsets
$\mathrm{Spec}\,R_0[t]$ and $\mathrm{Spec}\,R_0[t\inv]$ with
intersection $\mathrm{Spec}\,R_0[t,t\inv]$, where $t = xy^{-1}$. The
category of quasi-coherent sheaves of $\cO_{\pp}$\nbd-modules on the
projective line is equivalent to a category of diagrams of the form
\begin{displaymath}
  M^- \rTo M^0 \lTo M^+
\end{displaymath}
with entries modules over the rings $R_0[t\inv]$, $R_0[t,t\inv]$
and~$R_0[t]$, respectively, subject to a certain gluing
condition. Even when $R_0$ fails to be commutative we can {\it
  define\/} a quasi-coherent sheaf on the projective line to be such a
diagram. This generalisation has been considered in the context of
algebraic $K$-theory by \textsc{Bass} \cite[\S{}XII.9]{MR0249491} and
\textsc{Quillen} \cite[\S{}8]{MR0338129}, who showed that there is a
splitting of \textsc{Quillen} $K$-groups
\begin{equation*}
  K_q (\pp) \iso K_q (R_0) \oplus K_q (R_0) \tag{\(\dagger\)}
\end{equation*}
for all $q \geq 0$, where the left-hand side denotes the $K$-groups of
a category of diagrams as above subject to suitable finiteness
conditions. One can even go further and replace the (\textsc{Laurent})
polynomial rings by twisted (\textsc{Laurent}) polynomial rings with
respect to an automorphism of~$R_{0}$. This has been observed by
\textsc{Grayson} \cite[Theorem~1.1]{MR929766}, and an explicit proof
containing full technical details has been made available by
\textsc{Yao} \cite[\S1]{MR1325783}.

\medbreak

A more drastic generalisation may be attempted, starting from the
observation that the \textsc{Laurent} polynomial ring $R_0[t,t\inv]$
is a $\bZ$-graded ring, and that the polynomial rings $R_0[t\inv]$ and
$R_0[t]$ are the subrings generated by elements of non-positive
\resp{} non-negative degrees. It will be shown in this paper that the
$K$\nbd-theoretical splitting result above is still valid in this
generalised situation, provided the $\bZ$-graded ring in question is
in fact \textit{strongly} $\bZ$-graded. Strong gradings were
introduced by \textsc{Dade} \cite{MR593823} to capture and generalise
the quint\-essential properties of group rings; as in the original
representation-theoretic applications, it is the strongly graded
structure rather than the special form of group rings that controls
the splitting behaviour of algebraic $K$-theory.

\medbreak

It might be worthwhile to illustrate the gain in generality by an
explicit example.  For $K$ a field the (commutative) ring
\begin{equation*}
  \R = K[A,B,C,D] \big/ (AB+CD-1) \ ,
  \tag{\(\ddagger\)}
\end{equation*}
graded by $\deg(A) = \deg(C) = 1$ and $\deg(B) = \deg(D) = -1$, is a
strongly $\bZ$\nbd-graded ring; it can be shown that $\R$ contains
only the trivial units $K^{\times} \subset \R$ so that $\R$~is not a
crossed product, and certainly not a (twisted) \textsc{Laurent}
polynomial ring. Other examples of strongly $\bZ$-graded rings include
\textsc{Leavitt} algebras of type~$(1,n)$, or more generally
\textsc{Leavitt} path algebras associated with finite graphs without
sink (\textsc{Hazrat} \cite[Theorem~3.15]{MR3096577}).  The splitting
result~($\dagger$) remains valid for the projective lines associated
with these rings.

\medbreak

The proof follows broadly the pattern laid out by \textsc{Quillen}
\cite[\S{}8]{MR0338129}, formulated in terms of \textsc{Waldhausen}'s
$\mathcal{S}_\mathbf{\cdot}$\nbd-construction \cite{MR802796}. One
difference to note is that we eschew the use of regular sheaves in
favour of a simpler notion of vector bundles without higher sheaf
cohomology. Moreover, there are major deviations from the template in
several places necessitated by the much larger class of rings under
consideration. The increased generality leads to unexpected
phenomena. For example, we encounter a double-indexed family of
non-isomorphic twisting sheaves $\OO k \ell$ on the projective line in
place of the more familiar single-indexed family
$\mathcal{O} (k + \ell)$; the resulting twist operation leads, in
effect, to a $\bZ$\nbd-action on the collection of cohomology modules
of a given sheaf (Theorem~\ref{thm:cohomology_of_twist}).

Compared to the existing treatments of twisted \textsc{Laurent}
polynomial rings the graded techniques are simple and elegant,
avoiding all unpleasantries with changing module structures by the
prescribed ring automorphism. This indicates that strongly graded
rings are the ``right'' generalisation of groups rings, even in
contexts not originally considered by \textsc{Dade}.

\subsection*{Motivation}

This paper is part of the second author's PhD thesis. The results are
crucial for establishing the fundamental theorem for the algebraic
$K$\nbd-theory of strongly $\bZ$\nbd-graded rings
(\textsc{H\"uttemann} \cite{huettemann2020fundamental}), relating the
algebraic $K$\nbd-groups of $\R = \bigoplus_{k \in \bZ} \R_{k}$ with
subgroups and quotients of the algebraic $K$\nbd-groups of~$\R_{0}$
determined by the graded structure. The projective line has also been
used by \textsc{Steers} and the first author to promote
results on algebraic finite domination of chain complexes from
\textsc{Laurent} polynomial rings to the class of strongly
$\bZ$\nbd-graded rings \cite{HS}.

\subsection*{Acknowledgements}

The idea to consider the projective space associated with a strongly
$\bZ$\nbd-graded ring was developed during a research visit of the
first author to Beijing Institute of Technology. Their hospitality and
financial support are greatly appreciated.

\part{On strongly $\bZ$-graded rings}

\section{Conventions and notation used in the paper}
\label{sec:conventions}

A ``ring'' will always mean an associative unital ring. The term
``module'', if unspecified, always refers to a unital right module
over the ring in question.

Let $\R = \bigoplus_{\bZ} R_{k}$ be a unital $\bZ$-graded ring. As
usual, the elements of~$R_{k}$ are called the {\it homogeneous
  elements of degree~$k$}. It is known that $R_{0}$ is a subring
of~$\R$ with the same unit element as~$\R$
\cite[Proposition~1.4]{MR593823}.  The two sets
\begin{displaymath}
  \Rn = \bigoplus_{n \leq 0} R_{n} \qquad \text{and} \qquad \Rp =
  \bigoplus_{n \geq 0} R_{n}
\end{displaymath}
are $\bZ$\nbd-graded subrings of~$\R$. We write $\Rnn k$ for the
$\Rn$\nbd-module $\bigoplus_{n \leq k} R_{n}$, and $\Rpn {-\ell}$ for
the $\Rp$\nbd-module $\bigoplus_{n \geq -\ell} R_{n}$.

The reader may want to keep the following motivating example in mind:
$R_{0}$ a commutative unital ring, $\R = R_{0}[t,t\inv]$ the
\textsc{Laurent} polynomial ring in one indeterminate,
$\Rn = R_{0}[t\inv]$ and $\Rp=R_{0}[t]$ the polynomial rings. Then
$\Rpn {-\ell} = t^{-\ell} R_{0}[t]$ and $\Rnn k = t^{k} R_{0}[t\inv]$.
--- A more complicated example, requiring the full force of strongly
graded techniques to be introduced presently, is the graded ring
$R = K [A,B,C,D]/(AB+CD-1)$ introduced in~(\(\ddagger\)) above.

\section{Strongly graded rings}
\label{sec:basic_properties}

Let $\R = \bigoplus_{k \in \bZ} R_{k}$ be a $\bZ$\nbd-graded ring.  A
{\it \pou k} is a finite sum expression
$1 = \sum_{j} \lambda_{j} \rho_{j}$ with $\lambda_{j} \in R_{k}$ and
$\rho_{j} \in R_{-k}$. Partitions of unity may or may not exist in
general. --- Following \textsc{Dade} \cite{MR593823} we call $\R$ a
{\it strongly $\bZ$-graded ring} if $R_{k} R_{\ell} = R_{k+\ell}$ for
all $k, \ell \in \bZ$, where the left-hand side consists of all finite
sums of products $ab$ with $a \in R_{k}$ and $b \in R_{\ell}$.

\begin{proposition}[Characterisation of strongly graded rings;
  see \textsc{Dade} {\cite[Proposition~1.6]{MR593823}}, and also
  \textsc{H\"uttemann} and \textsc{Steers}
  {\cite[Proposition~1.5]{HS}}]
  \label{prop:characterisation_strongly_graded}
  The following statements are equivalent:
  \begin{enumerate}
  \item \label{item:1} The $\bZ$\nbd-graded ring \(\R\) is strongly graded.
  \item \label{item:2} For every \(k\in\bZ\) there is at least one
    partition of unity of type \((k,-k)\).
  \item \label{item:3} There is at least one \pou 1, and at least one
    \apou 1.\qed
  \end{enumerate}
\end{proposition}

We now discuss the basic properties of strongly graded rings. Though
most are well known, we include some proofs to emphasise the benefits
of strong grading.

\begin{lemma}
  \label{lem:projective_modules}
  Let $\R$ be a strongly $\bZ$\nbd-graded ring, and let $n \in
  \bZ$.
  The homogeneous component $R_{n}$ is a finitely generated projective
  $R_{0}$\nbd-module. The $\Rp$\nbd-module $\Rpn {-n}$ is finitely
  generated projective, as is the $\Rn$\nbd-module $\Rnn n$. These
  statements are true for both left and right module versions.
\end{lemma}

\begin{proof}
  We will deal with the second statement (right module version only);
  the first part has an almost identical proof, the third part is true
  by symmetry. Let $1 = \sum_{j} \lambda_{j} \rho_{j}$ be a \apou
  n. Then we have right $\Rp$\nbd-linear maps
  \begin{displaymath}
    \beta_{j} \colon \Rpn {-n} \rTo \Rp \ , \quad x \mapsto \rho_{j} x
  \end{displaymath}
  such that $\sum_{j} \lambda_{j} \beta_{j}$ is the identity map
  of~$\Rpn {-n}$. By the dual basis theorem \cite[\S{}II.2 No.~6,
  Proposition~12]{MR1727844} this means that $\Rpn {-n}$ is a finitely
  generated projective $\Rp$\nbd-module (with generators~$\lambda_{j}$).
\end{proof}

\begin{corollary}
  \label{cor:flat}
  Suppose that $\R$ is a strongly $\bZ$\nbd-graded ring. Then $R$ is a
  flat left $\Rp$-module, and a flat left $\Rn$-module.
\end{corollary}

\begin{proof}
  The left $\Rp$-modules $\Rpn {-n}$ are projective, by
  Lemma~\ref{lem:projective_modules}, hence their nested union
  $\bigcup_{n \geq 0} \Rpn {-n} = R$ is flat. The other statement has
  a symmetric proof.
\end{proof}

\begin{lemma}
  \label{lem:bimodules}  
  Let $\R$ be a strongly $\bZ$\nbd-graded ring, and let
  $k, \ell \in \bZ$. The multiplication map
  $\mu \colon a \tensor b \mapsto ab$ yields isomorphisms
  \begin{align}
    \Rnn \ell \tensor_{R_{0}} R_k & \iso \Rnn {\ell+k} && \text{(as
      \(\Rn\)-\(R_{0}\)-bimodules),} \label{iso:neg0-n} \\
    \Rnn k \tensor_{\Rn} \R & \iso \R && \text{(as
      \(\Rn\)-\(\R\)-bimodules),} \label{iso:negn-R} \\
    \Rnn \ell \tensor_{\Rn} \Rnn k & \iso \Rnn {\ell+k} && \text{(as
      \(\Rn\)-\(\Rn\)-bimodules),} \label{iso:negm-negn} \\
    R_{k} \tensor_{R_{0}} R_{\ell} & \iso R_{k+\ell} && \text{(as
      \(R_{0}\)-\(R_{0}\)-bimodules),} \label{iso:m-n} \\
    \Rpn {-k} \tensor_{\Rp} \Rpn {-\ell} & \iso \Rpn {-k-\ell} &&
    \text{(as \(\Rp\)-\(\Rp\) bimodules),} \label{iso:posm-posn} \\
    \Rpn  {-\ell} \tensor_{\Rp} \R & \iso \R && \text{(as
      \(\Rp\)-\(\R\)-bimodules),} \label{iso:posn-R}\\
    \Rpn {-k} \tensor_{R_{0}} R_{-\ell} & \iso \Rpn {-k-\ell} && \text{(as
      \(\Rp\)-\(R_{0}\)-bimodules).} \label{iso:pos0-n}
  \end{align}
  In particular, $\Rnn k$ is an invertible $\Rn$-bimodule with
  inverse~$\Rnn {-k}$ by~\eqref{iso:negm-negn}, $R_{k}$ is an
  invertible $R_{0}$\nbd-bi\-module with inverse~$R_{-k}$
  by~\eqref{iso:m-n}, and $\Rpn {-k}$ is an invertible $\Rp$-bimodule
  with inverse~$\Rpn k$ by~\eqref{iso:posm-posn}.
\end{lemma}

\begin{proof}
  We prove \eqref{iso:neg0-n} only, the other statements have similar
  proofs. Let $1 = \sum_{j} x_{j} y_{j}$ be a \apou k. The map
  \begin{displaymath}
    \rho \colon \Rnn {\ell + k} \rTo \Rnn \ell \tensor_{R_{0}} R_k \ , \quad r
    \mapsto \sum_j rx_j \tensor y_j
  \end{displaymath}
  is an \(\Rn\)-\(R_{0}\)-bimodule
  map; indeed, linearity on the left is immediate from the definition,
  while for $r \in \Rnn {\ell + k}$ and $s \in R_{0}$ we calculate
  (using $y_k s x_j \in R_{0}$ at the third equality sign)
  \begin{multline*}
    \rho(rs) = \sum_j rsx_j \tensor y_j = \sum_j \sum_k rx_k y_k s x_j
    \tensor y_j \\ = \sum_k \sum_j r x_k \tensor y_k s x_j y_j = \sum_k r
    x_k \tensor y_k s = \rho(r) \cdot s \ .
  \end{multline*}
  We have $\mu \rho (r) = \sum_j r x_j y_j = r$, and also, for
  $r \in \Rnn {\ell}$ and $s \in R_k$,
  \begin{displaymath}
    \rho \mu (r \tensor s) = \rho (rs) = \sum_j rs x_j \tensor y_j =
    \sum_j r \tensor s x_j y_j = r \tensor s
  \end{displaymath}
  so that $\mu$ is an isomorphism with inverse~$\rho$.
\end{proof}

\begin{proposition}
  \label{prop:0-all}
  Let $\R$ be a strongly $\bZ$\nbd-graded ring, let $M$ be an
  $\Rn$-module, and let $k \in \bZ$. The map
  \begin{displaymath}
    \Omega_{M} \colon M \tensor_{R_{0}} R_{k} \rTo M \tensor_{\Rn} \Rnn k
    \ , \quad x \tensor y \mapsto x \tensor y
  \end{displaymath}
  of right $R_{0}$\nbd-modules is an isomorphism. Similarly, there are
  isomorphisms of right $R_{0}$-modules
  $N \tensor_{R_{0}} R_{-\ell} \rTo N \tensor_{\Rp} \Rpn {-\ell}$ and
  $Q \tensor_{R_{0}} R_{k} \rTo Q \tensor_{\R} \R$ for any
  $\Rp$\nbd-module~$N$ and any $\R$\nbd-module~$Q$, and any integer
  $\ell \in \bZ$.
\end{proposition}

\begin{proof}
  For $M = \Rn$ this is true since both source and target are
  isomorphic, {\it via\/} the multiplication map
  $m \tensor r \mapsto mr$, to $\Rnn k$ by~\eqref{iso:neg0-n}
  and~(\ref{iso:negm-negn}). By taking direct sums the claim follows
  for $M$ a free $\Rn$\nbd-module (with a specified basis, possibly
  infinite). The general case follows from considering a free
  $\Rn$\nbd-module presentation of~$M$: choose free modules $G$
  and~$F$ and maps so as to obtain an exact sequence
  $G \rTo F \rTo M \rTo 0$, and produce from this a commutative
  diagram
  \begin{diagram}
    G \tensor_{R_{0}} R_{k} & \rTo & F \tensor_{R_{0}} R_{k} & \rTo &
    M \tensor_{R_{0}} R_{k} & \rTo & 0 \\
    \dTo>{\Omega_{G}} && \dTo>{\Omega_{F}} && \dTo>{\Omega_{M}} \\
    G \tensor_{\Rnn 0} \Rnn k & \rTo & F \tensor_{\Rnn 0} \Rnn k &
    \rTo & M \tensor_{\Rnn 0} \Rnn k & \rTo & 0
  \end{diagram}
  which has exact rows since tensor products are right exact. We know
  already that $\Omega_{G}$ and~$\Omega_{F}$ are isomorphisms, hence
  so is~$\Omega_{M}$ by the five lemma. --- The remaining statements
  are proved in a similar fashion.
\end{proof}

\part{The projective line of a strongly $\bZ$-graded ring}

From now on, we assume for the remainder of the paper that
$\R = \bigoplus_{k \in \bZ} R_{k}$ is a strongly $\bZ$\nbd-graded
ring.

\section{Sheaves on the projective line}
\label{sec:sheaves}

\begin{definition}
  A {\it quasi-coherent sheaf on~$\pp$}, or just {\it sheaf} for
  short, is a diagram
  \begin{equation}
    \label{eq:sheaf}
    \cY = \quad \Big( Y^{-} \rTo^{\upsilon^{-}} Y^{0} \lTo^{\upsilon^{+}}
    Y^{+} \Big)
  \end{equation}
  where $Y^{-}$, $Y^{0}$ and $Y^{+}$ are modules over $\Rn$, $\R$
  and~$\Rp$, respectively, with an $\Rn$\nbd-linear homomorphisms
  $\upsilon^{-}$ and an $\Rp$\nbd-linear homomorphism $\upsilon^{+}$,
  such that the diagram of the adjoint $\R$\nbd-linear maps
  \begin{equation}
    \label{eq:sheaf_cond}
    Y^{-} \tensor_{\Rn} \R \rTo[l>=3em]^{\upsilon^{-}_\sharp}_{\iso} Y^{0}
    \lTo[l>=3em]^{\upsilon^{+}_{\sharp}}_{\iso} Y^{+} \tensor_{\Rp} \R
  \end{equation}
  consists of isomorphisms. This latter condition will be referred to
  as the \textit{sheaf condition}.
\end{definition}

These quasi-coherent sheaves form a category $\qco$ in an obvious way;
a morphism $f \colon \cY \rTo \cZ$ is a triple
$f = (f^{-}, f^{0}, f^{+})$ of linear maps (of modules over different
rings) compatible with the structure maps $\upsilon^{\pm}$ of~$\cY$
and the structure maps~$\zeta^{\pm}$ of~$\cZ$ in the sense that
$f^{0} \circ \upsilon^{\pm} = \zeta^{\pm} \circ f^{\pm}$. In fact,
{\it $\qco$ is an \textsc{abel}ian category}. Kernels and cokernels
are computed pointwise as in diagram categories; the only non-trivial
thing to verify is that, given a morphism $f \colon \cY \rTo \cZ$ of
sheaves, the diagram
\begin{displaymath}
  \ker (f) = \ \Big( \ker (f^{-}) \rTo^{\kappa^{-}} \ker (f^{0})
  \lTo^{\kappa^{+}} \ker (f^{+}) \Big)
\end{displaymath}
is a sheaf itself. This follows from the five lemma, applied to the
diagram below.
\begin{diagram}[textflow]
  0 & \rTo & \ker (f^{-}) \tensor_{\Rn} \R & \rTo & Y^{-}
  \tensor_{\Rn} \R & \rTo & Z^{-} \tensor_{\Rn} \R \\
  && \dTo>{\kappa^{-}} && \dTo>{\iso} && \dTo>{\iso} \\
  0 & \rTo & \ker (f^{0}) & \rTo & Y^{0} & \rTo & Z^{0}
\end{diagram}
The top row is exact as $\R$ is a flat left $\Rn$\nbd-module by
Corollary~\ref{cor:flat}; the second and third vertical maps are
isomorphisms as $\cY$ and~$\cZ$ are sheaves. This forces $\kappa^{-}$
to be an isomorphism as well. The argument works {\it mutatis
  mutandis\/} for~$\kappa^{+}$. --- Working in an
\textsc{abel}ian category we have a canonical notion of exactness at
our disposal: a sequence of sheaves and morphisms
$0 \rTo \cX \rTo^{g} \cY \rTo^{f} \cZ \rTo 0$ is (short) exact if for
all decorations $? \in \{-,0,+\}$ the sequence
\begin{displaymath}
  0 \rTo X^{?} \rTo^{g^{?}} Y^{?} \rTo^{f^{?}} Z^{?} \rTo 0
\end{displaymath}
is an exact sequence of $R_{0}$\nbd-modules.

\section{Vector bundles and twisting sheaves}
\label{sec:twisting}

\begin{definition}
  The sheaf $\cY$ of~\eqref{eq:sheaf} is called {\it locally finitely
    generated\/} if its constituent modules are finitely generated as
  modules over their respective ground rings. We call $\cY$ a {\it
    vector bundle\/} if its constituent modules are finitely generated
  projective modules over their respective ground rings. The category
  of vector bundles (and all morphisms of sheaves between them) is
  denoted by~$\vb$.
\end{definition}

Important examples of vector bundles are the {\it twisting sheaves\/}
\begin{displaymath}
  \OO k \ell = \quad \Big( \Rnn k \rTo^{\subseteq} \R \lTo^{\supseteq}
  \Rpn {-\ell} \Big)
\end{displaymath}
where $k$ and~$\ell$ are integers. These objects are sheaves by
Lemma~\ref{lem:bimodules}, and are indeed vector bundles by
Lemma~\ref{lem:projective_modules}. For us $\OO k \ell$ plays the
r\^ole of what is denoted by $\mathcal{O}_{\pp} (k + \ell)$ in
algebraic geometry. It should be emphasised that in general the
twisting sheaves as defined here depend on both $k$ and~$\ell$ and not
just on the sum $k + \ell$. If, however, $\R$~is a crossed product
(equivalently, if $\R$ contains an invertible homogeneous element of
degree~$1$) then the sum $k+\ell$ determines $\OO k \ell$ up to
isomorphism in~$\vb$.

\begin{proposition}
  \label{prop:cartesian_square}
  Let $k, \ell \in \bZ$. The commuting square of twisting sheaves and
  inclusion maps
  \begin{diagram}
    \OO {k} {\ell} & \rTo^{\rho} & \OO {k} {\ell+1} \\
    \dTo<{\lambda} && \dTo<{\lambda} \\
    \OO {k+1} {\ell} & \rTo[l>=3em]^{\rho} & \OO {k+1} {\ell+1}
  \end{diagram}
  is cartesian in the sense that there results a short exact sequence
  of sheaves
  \begin{multline}
    \label{eq:cartesian}
    0 \rTo \OO {k} {\ell} \rTo[l>=3em]^{\left( {\lambda \atop \rho} \right)}
    \OO {k+1} {\ell} \oplus \OO {k} {\ell+1} \\ \rTo[l>=4em]^{(-\rho, \lambda)}
    \OO {k+1} {\ell+1} \rTo 0 \ .
  \end{multline}
\end{proposition}

\begin{proof}
  Restricting to ``$-$''-components, the sequence~\eqref{eq:cartesian}
  becomes the sequence of $\Rn$\nbd-modules
  \begin{displaymath}
    0 \rTo \Rnn k \rTo^{\left( {1 \atop 1} \right)} \Rnn {k+1} \oplus
    \Rnn k \rTo^{(-1 \ 1)} \Rnn {k+1} \rTo 0
  \end{displaymath}
  which is clearly exact. Similarly for the other two components.
\end{proof}

\begin{definition}
  Let $\cY$ be a sheaf, and let $k, \ell \in \bZ$. We define the
  {\it $(k,\ell)${\rm th} twist of~$\cY$}, denoted $\tw \cY k \ell$,
  to be the sheaf
  \begin{displaymath}
    \tw \cY k \ell = \quad \Big( Y^{-} \tensor_{\Rn} \Rnn k \rTo Y^{0} 
    \tensor_{R} R \lTo Y^{+} \tensor_{\Rp} \Rpn {-\ell} \Big) \ ,
  \end{displaymath}
  with structure maps induced by those of~$\cY$ and the inclusion
  maps.
\end{definition}

One should maybe think of $\cY(k,\ell)$ as the ``($k+\ell$)th twist
of~$\cY$'' in the sense of algebraic geometry. --- Twisting is
functorial, and is in fact exact: a short exact sequence of sheaves
\begin{displaymath}
  0 \rTo \cX \rTo \cY \rTo \cZ \rTo 0
\end{displaymath}
is transformed into another short exact sequence
\begin{displaymath}
  0 \rTo \tw \cX k \ell \rTo \tw \cY k \ell \rTo \tw \cZ k \ell \rTo 0
  \ .
\end{displaymath}
This is true as each entry in the twisted version is obtained from the
untwisted one by taking the tensor product with a projective, hence
flat, module over the relevant ring,
cf.~Lemma~\ref{lem:projective_modules}. Moreover, there are
isomorphisms, natural in the sheaf~$\cY$,
\begin{equation}
  \label{eq:twisting}
  \Tw {\tw \cY k \ell} {k'} {\ell'} \iso \tw \cY {k+k'} {\ell + \ell'}
  \quad\text{and}\quad \tw \cY 0 0 \iso \cY \ ,
\end{equation}
as follows immediately from the isomorphisms listed in
Lemma~\ref{lem:bimodules}.

\begin{lemma}
  \label{lem:cF_onto}
  Let $\cY$ be a locally finitely generated sheaf. Then there exists a
  vector bundle $\cF$ and an epimorphism $\cF \rTo \cY$, where $\cF$
  is a finite direct sum of twisting sheaves $\OO k \ell$ with $k,\ell
  \leq 0$ and $k+\ell < 0$.
\end{lemma}

\begin{proof}
  Let $y \in Y^{0}$ be an arbitrary element. As tensor products
  commute with nested unions, and as $\cY$ satisfies the sheaf
  condition~\eqref{eq:sheaf_cond}, we can consider $y$ as an element
  in the \textsc{abel}ian group
  \begin{displaymath}
    Y^{0} \iso Y^{+} \tensor_{\Rp} \R = Y^{+} \tensor_{\Rp}
    \bigcup_{n^{+}  \geq 0} \Rpn {-n^{+}} = \bigcup_{n^{+} \geq 0}
    Y^{+} \tensor_{\Rp} \Rpn {-n^{+}} \ ;
  \end{displaymath}
  that is, we find $n^{+} > 0$ and an element
  $y^{+} \in Y^{+} \tensor_{\Rp} \Rpn {-n^{+}}$ such that $y$ is the
  image of~$y^{+}$ under the map
  $z \tensor r \mapsto \upsilon^{+}(z) r$. Similarly, we find
  $n^{-} > 0$ and an element
  $y^{-} \in Y^{-} \tensor_{\Rn} \Rnn {n^{-}}$ such that $y$ is the
  image of~$y^{-}$ under the map
  $z \tensor r \mapsto \upsilon^{-}(z) r$. We thus get a map of
  sheaves
  $g = (g^{-}, g^{0}, g^{+}) \colon \OO 0 0 \rTo \tw {\cY} {n^{-}}
  {n^{+}}$
  determined by sending the elements~$1$ of $\Rn$, $\R$ and~$\Rp$ to
  $y^{-}$, $y$ and~$y^{+}$, respectively. Write $k = -n^{-} < 0$ and
  $\ell = -n^{+} < 0$; by twisting we obtain a map of sheaves
  \begin{displaymath}
    g(k, \ell) = f = (f^{-}, f^{0}, f^{+}) \colon \OO k \ell \rTo 
    \tw {\tw {\cY} {n^{-}} {n^{+}}} {k} {\ell} \iso \cY
  \end{displaymath}
  such that $y$ is in the image of the homomorphism $f^{0} =
  g^{0}$.
  As $Y^{0}$ is finitely generated as an $\R$\nbd-module, it follows
  that there is a finite direct sum as in the statement of the lemma
  mapping surjectively onto the $Y^{0}$-component.

  Similarly, given $y^{+} \in Y^{+}$ we find $n^{-} > 0$ and
  $y^{-} \in Y^{-} \tensor_{\Rn} \Rnn {n^{-}}$ such that
  $\upsilon^{+}(y^{+})$ is the image of~$y^{-}$ under the map
  $z \tensor r \mapsto \upsilon^{-}(z) r$. We thus obtain a map of
  sheaves $\OO 0 0 \rTo \tw {\cY} {n^{-}} 0$ by sending the
  elements~$1$ to $y^{-}$, $\upsilon^{+}(y^{+})$ and~$y^{+}$,
  respectively. By twisting, this yields a map of sheaves
  $\OO k 0 \rTo \cY$, where $k = -n^{-} < 0$, such that $y^{+}$ is in
  the image of the map of $+$\nbd-components. As $Y^{+}$ is finitely
  generated as an $\Rp$\nbd-module, it follows that there is a finite
  direct sum as in the statement of the lemma mapping surjectively
  onto the $Y^{+}$-component.

  This argument works {\it mutatis mutandis\/ for the
    ``$-$''\nbd-com\-ponent }. Combining the three resulting
  surjections thus constructed yields the requisite map of sheaves.
\end{proof}

\section{Sheaf cohomology}
\label{sec:cohomology}

\begin{definition}
  Given a sheaf
  $\cY = \Big( Y^{-} \rTo^{\upsilon^{-}} Y^{0} \lTo^{\upsilon^{+}}
  Y^{+} \Big)$
  we define its (\textsc{\v Cech}) cohomology $R_{0}$\nbd-modules by
  setting $\hh q \cY = 0$ for $q \geq 2$, and by
  \begin{displaymath}
    \hh 0 \cY = \ker (\upsilon^{-} - \upsilon^{+}) \qquad \text{and}
    \qquad \hh 1 \cY = \coker (\upsilon^{-} - \upsilon^{+}) \ .
  \end{displaymath}
\end{definition}

We thus have a canonical exact sequence of $R_{0}$\nbd-modules
\begin{equation}
  \label{eq:canonical}
  0 \rTo \hh 0 \cY \rTo Y^{-} \oplus Y^{+}
  \rTo[l>=4em]^{\upsilon^{-}-\upsilon^{+}} Y^{0} \rTo \hh 1 \cY \rTo 0
\end{equation}
which defines the cohomology modules up to canonical isomorphism.
Note that $\hh q \cY = \invlim{}^{\!q}
\cY$ for all $q \geq 0$, where we consider $\cY$ as a diagram of
$R_{0}$\nbd-modules. Consequently, a short exact sequence
\begin{displaymath}
  0 \rTo \cX \rTo \cY \rTo \cZ \rTo 0
\end{displaymath}
of sheaves yields a not-very-long exact sequence of
$R_{0}$\nbd-modules
\begin{multline}
  \label{eq:not-so-long}
  0 \rTo \hh 0 \cX \rTo \hh 0 \cY \rTo \hh 0 \cZ \\ \rTo \hh 1 \cX
  \rTo \hh 1 \cY \rTo \hh 1 \cZ \rTo 0\ .
\end{multline}
We may refer to $\hh 0 \cY$ occasionally as the {\it global
  sections\/} of~$\cY$.

By direct inspection we obtain the following calculation of cohomology
modules of the twisting sheaves:

\goodbreak

\begin{proposition}
  \label{prop:H_of_On}
  Let $k, \ell \in \bZ$.
  \begin{enumerate}
  \item If $k + \ell \geq 0$ then $\Hh 1 {\OO k \ell}$ is trivial, and
    there is an isomorphism
    $\Hh 0 {\OO k \ell} \iso \bigoplus_{n=-\ell}^{k} R_{n}$.
  \item If $k + \ell < 0$ then $\Hh 0 {\OO k \ell}$ is trivial, and
    there is an isomorphism $\Hh 1 {\OO k \ell} \iso
    \bigoplus_{n=k+1}^{-\ell-1} R_{n}$.
  \end{enumerate}
  In particular, $\Hh 0 {\OO k \ell}$ and $\Hh 1 {\OO k \ell}$ are
  finitely generated projective $R_{0}$\nbd-modules.\qed
\end{proposition}

We now exhibit a phenomenon that does not occur in algebraic geometry:
given an integer $n$ and a sheaf~$\cY$ there are infinitely many
``$n$th twists'' of~$\cY$, \viz, the sheaves $\tw \cY k \ell$ with
$k + \ell = n$. This leads to a non-trivial action of~$\bZ$ on the
collection of cohomology modules of $n$th twists of~$\cY$:

\begin{theorem}
  \label{thm:cohomology_of_twist}
  Let $\cY$ be a sheaf. For any $k \in \bZ$ there is an isomorphism of
  $R_{0}$\nbd-modules, natural in~$\cY$,
  \begin{displaymath}
    \Hh q {\tw \cY k {-k}} \iso \hh q \cY \tensor_{R_{0}} R_{k} \ ,
    \quad q = 0,1 \ .
  \end{displaymath}
\end{theorem}

\begin{proof}
  Since $R_{k}$ is a projective $R_{0}$\nbd-module
  (Lemma~\ref{lem:projective_modules}), the defining exact
  sequence~\eqref{eq:canonical} for~$\hh q \cY$ remains exact after
  application of $\nix \tensor_{R_{0}} R_{k}$, resulting in the exact
  sequence of $R_{0}$-modules
  \begin{multline*}
    0 \rTo \hh 0 \cY \tensor_{R_{0}} R_{k} \rTo \big( Y^{-}
    \tensor_{R_{0}} R_{k} \big) \,\oplus\, \big( Y^{+} \tensor_{R_{0}}
    R_{k} \big) \rTo \\%
    \rTo Y^{0} \tensor_{R_{0}} R_{k} \rTo \hh 1 \cY \tensor_{R_{0}}
    R_{k} \rTo 0 \ .
  \end{multline*}
  The middle terms are isomorphic, by Proposition~\ref{prop:0-all}, to
  the modules
  \begin{displaymath}
    \big( Y^{-} \tensor_{\Rn} \Rnn k \big) \oplus \big( Y^{+}
    \tensor_{\Rp} \Rpn k \big) \quad \text{and} \quad Y^{0}
    \tensor_{\R} \R \ ,
  \end{displaymath}
  respectively, so that we actually recover the defining exact
  sequence for the cohomology of~$\tw \cY k {-k}$. In other words, 
  $\Hh q {\tw \cY k {-k}} \iso \hh q \cY \tensor_{R_0} R_{k}$ as
  claimed.
\end{proof}

\begin{corollary}
  \label{cor:cohomology_zero}
  Let $q \in \{0,1\}$, and suppose $k,\ell, k', \ell' \in \bZ$ are
  such that $k+\ell = k'+\ell'$. For any sheaf $\cY$ we have an
  equivalence
  \begin{displaymath}
    \Hh q {\tw \cY k \ell} = 0 \quad \Longleftrightarrow \quad \Hh q
    {\tw \cY {k'} {\ell'}} = 0 \ .
  \end{displaymath}
\end{corollary}

\begin{proof}
  In view of the isomorphisms~(\ref{eq:twisting}) we may assume,
  without loss of generality, that $k + \ell = 0 = k' = \ell'$ (in
  effect, we replace $\cY(k', \ell')$ with~$\cY$ in the
  statement). Then by the previous Theorem
  $\Hh q {\tw \cY k \ell} = 0$ if and only if
  $\hh q \cY \tensor_{R_0} R_{k} = 0$, which happens if and only if
  $\hh q \cY = 0$ as $R_{k}$ is an invertible $R_{0}$\nbd-bimodule by
  Lemma~\ref{lem:bimodules}.
\end{proof}

\section{Vector bundles with trivial first cohomology}

For $K$\nbd-theoretical considerations we will restrict attention to
those
vector bundles~$\cY$ having trivial first cohomology for all their
twists $\tw \cY k \ell$ with
$k+\ell \geq 0$: their global sections
are finitely generated projective $R_{0}$\nbd-modules. We will
demonstrate this in Theorem~\ref{thm:crucial_finiteness}, after
dispensing with some auxiliary results.

\begin{lemma}
  \label{lem:H1-zero-H0-projective}
  Let $\cY$ be a sheaf consisting of projective modules with vanishing
  $\hh 1 \cY$. Then $\hh 0 \cY$ is a projective $R_{0}$\nbd-module.
\end{lemma}

\begin{proof}
  In view of the hypothesis $\hh 1 \cY = 0$, the exact
  sequence~\eqref{eq:canonical} reduces to a short exact sequence of
  $R_{0}$\nbd-modules
  \begin{displaymath}
    0 \rTo \hh 0 \cY \rTo Y^{-} \oplus Y^{+} \rTo Y^{0} \rTo 0 \ .
  \end{displaymath}
  As $\R$ is strongly $\bZ$\nbd-graded, the rings $\Rn$, $\R$
  and~$\Rp$ are projective $R_{0}$\nbd-mod\-ules
  (Lemma~\ref{lem:projective_modules}). We conclude that projective
  modules over any of these rings are projective as
  $R_{0}$\nbd-modules; this applies in particular to $Y^{-}$, $Y^{0}$
  and~$Y^{+}$. Thus the above short exact sequence splits, and
  $\hh 0 \cY$ is a direct summand of the projective
  $\R_{0}$\nbd-module $Y^{-} \oplus Y^{+}$. Hence $\hh 0 \cY$ is a
  projective $\R_{0}$\nbd-module as well.
\end{proof}

\begin{lemma}
  \label{lem:H1_triv_high_twist}
  Let $\cY$ be a locally finitely generated sheaf. Then there exists
  $n_{0} \geq 0$ such that $\hh 1 {\tw \cY k \ell} = 0$ for all $k,
  \ell \in \bZ$ with $k+\ell \geq n_{0}$.
\end{lemma}

\begin{proof}
  By Lemma~\ref{lem:cF_onto} there exists an epimorphism
  $\cF \rTo \cY$, where $\cF$ is a finite sum of twisting sheaves
  $\OO {k_{j}} {\ell_{j}}$ with $k_{j}, \ell_{j} \leq 0$ and
  $k_{j}+\ell_{j} < 0$. Let $n_{0}$ be the maximum of the positive
  numbers $-k_{j}-\ell_{j}$ occurring here. We claim that this number
  has the required property.

  To see this, let $k, \ell \in \bZ$ with $k+\ell \geq n_{0}$ be
  given. As twisting is exact we then have an epimorphism
  $\epsilon \colon \tw \cF k \ell \rTo \tw \cY k
  \ell$. By~\eqref{eq:twisting} we have
  \begin{displaymath}
    \tw \cF k \ell = \bigoplus_{j} \tw {\OO {k_{j}} {\ell_{j}}} k \ell
    \iso \bigoplus_{j} \OO {k_{j}+k} {\ell_{j}+\ell}
  \end{displaymath}
  with $k_{j}+k+\ell_{j}+\ell \geq 0$, by choice of~$n_{0}$, so that
  $\tw \cF k \ell$ has trivial first cohomology by
  Proposition~\ref{prop:H_of_On}. By~\eqref{eq:not-so-long} we have an
  exact sequence
  \begin{displaymath}
    \ldots \rTo \underbrace{\Hh 1 {\tw \cF k \ell}}_{= 0} \rTo \Hh 1
    {\tw \cY k \ell} \rTo 0
  \end{displaymath}
  proving that $\Hh 1 {\tw \cY k \ell} = 0$ as advertised.
\end{proof}

\begin{lemma}
  \label{lem:H0_fg_proj}
  For every vector bundle $\cY$ there exists a number $n_{0} \geq 0$
  such that $\Hh 0 {\tw \cY k \ell}$ is a finitely generated
  projective $R_{0}$\nbd-module for all $k, \ell \in \bZ$ with
  $k+\ell \geq n_{0}$.
\end{lemma}

\begin{proof}
  By Lemma~\ref{lem:cF_onto} there exists a vector bundle
  $\cF = \bigoplus_{j} \OO {k_{j}} {\ell_{j}}$ and an epimorphism
  $\epsilon \colon \cF \rTo \cY$. Then $\cK = \ker \epsilon$ is a
  vector bundle as well, and we choose the integer~$n_{0}$, according
  to Lemma~\ref{lem:H1_triv_high_twist}, so that both $\cK$ and~$\cY$
  have trivial first cohomology when twisted accordingly.  Now suppose
  that $k+\ell \geq n_{0}$. By Lemma~\ref{lem:H1-zero-H0-projective}
  we know that $\Hh 0 {\tw \cY k \ell}$ is a projective
  $R_{0}$\nbd-module. From the short exact sequence
  \begin{displaymath}
    0 \rTo \tw \cK k \ell \rTo \tw \cF k \ell \rTo \tw \cY k \ell \rTo
    0
  \end{displaymath}
  and~\eqref{eq:not-so-long} we obtain an exact sequence of
  $R_{0}$\nbd-modules
  \begin{displaymath}
    \Hh 0 {\tw \cF k \ell} \rTo^{\alpha} \Hh 0 {\tw \cY k \ell} \rTo
    \Hh 1 {\tw \cK k \ell} = 0
  \end{displaymath}
  so that $\alpha$~is onto. But the domain of~$\alpha$ is a finitely
  generated projective $R_{0}$\nbd-module, by
  Proposition~\ref{prop:H_of_On}, so $\Hh 0 {\tw \cY k \ell}$ is
  finitely generated.
\end{proof}

\begin{definition}
  \label{def:vb_n}
  Let $n \in \bZ$. We denote by $\vb_{n}$ the full subcategory
  of~$\vb$ consisting of those vector bundles~$\cY$ with
  $\hh 1 {\tw \cY k \ell} = 0$ for all $k, \ell \in \bZ$ with
  $k + \ell \geq n$.
\end{definition}

We will see later that the algebraic $K$\nbd-theory of the
category~$\vb$ is the same as the algebraic $K$\nbd-theory
of~$\vb_{0}$. The point is that for~$\vb_{0}$ we have the following
crucial finiteness result at our disposal:

\begin{theorem}
  \label{thm:crucial_finiteness}
  For $\cY \in \vb_{0}$ and $k, \ell \in \bZ$ with $k+\ell \geq 0$,
  the $R_{0}$\nbd-module $\Hh 0 {\tw \cY k \ell}$ is finitely
  generated projective.
\end{theorem}

\begin{proof}
  Given integers $k,\ell \in \bZ$ with $k + \ell \geq 0$, and given
  any $j \geq 0$, we obtain from~\eqref{eq:cartesian}, by tensoring
  with~$\cY$, a short exact sequence of vector bundles
  \begin{multline*}
    0 \rTo \tw \cY k {\ell+j} \rTo \tw \cY {k+1} {\ell+j} \oplus \tw
    \cY k {\ell+j+1} \\ \rTo \tw \cY {k+1} {\ell+j+1} \rTo 0
  \end{multline*}
  which by~\eqref{eq:not-so-long} yields, 
  keeping in mind %
  that $\cY$ 
  is %
  an object in~$\vb_{0}$ 
  by hypothesis, %
  a short exact sequence
  \begin{multline*}
    0 \rTo \Hh 0 {\tw \cY k {\ell+j}} \\%
    \rTo \Hh 0 {\tw \cY {k+1} {\ell+j}} \oplus \Hh 0 {\tw \cY k
      {\ell+1+j}} \\%
    \rTo \Hh 0 {\tw \cY {k+1} {\ell+1+j}} \rTo 0 \ .
  \end{multline*}
  We choose $n_{0} \gg 0$ sufficiently large according to
  Lemma~\ref{lem:H0_fg_proj}, applied to the vector bundle~$\cY$.  We
  can now use downward induction on~$j$, starting with $j=n_{0}-1$. In
  each step, the second and third term are known to be finitely
  generated projective $R_{0}$\nbd-modules, whence the first term is
  so as well, for {\it all\/} integers $k, \ell$ with $k + \ell \geq
  0$. --- The last step of the induction is done with $j=0$, proving
  the claim.
\end{proof}

\goodbreak

In preparation for our $K$\nbd-theoretical deliberations we consider
cokernels and pushouts of certain maps of vector bundles.

\begin{lemma}
  \label{lem:coker_is_vb_n}
  Suppose that $\alpha \colon \cX \rTo \cY$ is an injective map of
  sheaves. Suppose further that $k, \ell \in \bZ$ are such that
  $\Hh 1 {\tw \cY k \ell} = 0$. Then the cohomology module
  $\Hh 1 {\Tw {\coker(\alpha)} k \ell}$ is trivial as well. --- In
  particular, if $\coker(\alpha)$ is a vector bundle, and if $\cY$ is
  an object of~$\vb_{n}$, then $\coker(\alpha)$~is an object
  of~$\vb_{n}$.
\end{lemma}

\begin{proof}
  As twisting is an exact endofunctor of the category~$\qco$, there is
  a canonical isomorphism
  $\Tw {\coker(\alpha)} k \ell \iso \coker\big( \tw \alpha k \ell
  \big)$ and a resulting short exact sequence of sheaves
  \begin{displaymath}
    0 \rTo \tw \cX k \ell \rTo[l>=4em]^{\tw \alpha k \ell} \tw \cY k \ell
    \rTo \Tw {\coker(\alpha)} k \ell \rTo 0 \ .
  \end{displaymath}
  The associated exact sequence in cohomology~\eqref{eq:not-so-long}
  ends in
  \begin{displaymath}
    \Hh 1 {\tw \cY k \ell} \rTo \Hh1 {\Tw {\coker(\alpha)} k \ell} 
    \rTo 0
  \end{displaymath}
  with first term being trivial, by hypothesis on~$\cY$. This shows
  that the second term is trivial as well.
\end{proof}

Consider now a commutative square diagram in the category~$\qco$:
\begin{diagram}[LaTeXeqno]
  \label{diag:square}
  \cX & \rTo^{\alpha} & \cY \\
  \dTo && \dTo \\
  \cZ & \rTo^{\beta} & \cW
\end{diagram}

\begin{lemma}
  \label{lem:pushouts_exist}
  Suppose that the diagram~\eqref{diag:square} is a pushout with
  $\alpha$ an injective map, and that the sheaves $\cZ$
  and~$\coker(\alpha)$ are vector bundles. Then the map~$\beta$ is
  injective, and both~$\cW$ and $\coker(\beta)$ are vector
  bundles. If, in addition, $\cY$ and~$\cZ$ are objects of~$\vb_{n}$
  then so is~$\cW$.
\end{lemma}

\begin{proof}
  As the square is a pushout, $\coker(\alpha)$ is isomorphic
  to~$\coker(\beta)$ so that the latter is a vector bundle. For each
  decoration $? \in \{-,0,+\}$ the short exact sequence
  \begin{displaymath}
    0 \rTo X^{?} \rTo^{\alpha^{?}} Y^{?} \rTo \coker(\alpha^{?}) \rTo 0
  \end{displaymath}
  splits as the third non-trivial term is a projective module; it
  follows that $\alpha^{?}$ is a split injection, whence its pushout
  $\beta^{?}$ is a split injection as well. Hence we obtain a split
  short exact sequence
  \begin{displaymath}
    0 \rTo Z^{?} \rTo^{\beta^{?}} W^{?} \rTo \coker(\beta^{?}) \rTo 0
  \end{displaymath}
  with first and third term finitely generated projective (as both
  $\cZ$ and~$\coker(\beta)$ are vector bundles). This forces $W^{?}$
  to be finitely generated projective whence $\cW$ is a vector bundle.

  Now suppose, in addition to the previous hypotheses, that $\cY$
  and~$\cZ$ are objects of $\vb_{n}$. From the exact sequence of
  sheaves
  \begin{displaymath}
    0 \rTo \tw \cZ k \ell \rTo[l>=4em]^{\tw \beta k \ell} \tw \cW k \ell
    \rTo \Tw {\coker(\beta)} k \ell \rTo 0
  \end{displaymath}
  (note that twisting is exact, and commutes with pushouts as tensor
  products preserve colimits) we construct an exact sequence snippet
  in cohomology
  \begin{displaymath}
    \Hh 1 {\tw \cZ k \ell} \rTo \Hh 1 {\tw \cW k \ell} \rTo 
    \Hh1 {\Tw {\coker(\beta)} k \ell} \ ,
  \end{displaymath}
  cf.~\eqref{eq:not-so-long}. Since $\cZ \in \vb_{n}$ by hypothesis,
  and since $\coker(\beta) \iso \coker (\alpha) \in \vb_{n}$ by
  Lemma~\ref{lem:coker_is_vb_n}, we conclude that the middle term is
  trivial whenever $k + \ell \geq n$ so that $\cW \in \vb_{n}$ as
  claimed.
\end{proof}

\part{The algebraic $K$-theory of the projective line}
\label{part:K(P1)}

We continue to assume that $\R = \bigoplus_{k \in \bZ} R_{k}$ is a
strongly $\bZ$\nbd-graded ring.

\section{Algebraic $K$-theory}

A map~$f$ of chain complexes of quasi-coherent sheaves will be called
a {\it $q$-equiv\-a\-lence}, or a {\it quasi-isomorphism\/}, if $f^{?}$
is a quasi-isomorphism of chain complexes of modules for each
decoration $? \in \{-,0,+\}$. We let $\ch \vb$ denote the category of
bounded chain complexes of vector bundles; similarly, we denote by
$\ch \vb_{n}$ the category of bounded chain complexes of vector
bundles in~$\vb_{n}$, cf.~Definition~\ref{def:vb_n}. The categories
$\ch \vb$ and $\ch \vb_{n}$ are \textsc{Waldhausen} categories with
weak equivalences the quasi-isomorphisms, and cofibrations the
injections with cokernel a complex of vector bundles. Existence of the
requisite pushouts has been verified above in
Lemma~\ref{lem:pushouts_exist}; note that in the case of~$\vb_{n}$ the
cokernel of a cofibration is an object of~$\ch\vb_{n}$ again by
Lemma~\ref{lem:coker_is_vb_n}. --- As weak equivalences are defined
homologically, they satisfy the saturation and extension axioms. All
categories mentioned have a cylinder functor given by the usual
mapping cylinder construction which satisfies the cylinder axiom.

\begin{definition}
  \label{def:K_P1}
  The $K$-theory space of the projective line is defined to be
  \begin{displaymath}
    K(\pp) = \Omega |q\mathcal{S}_{\bullet} \ch\vb | \ ,
  \end{displaymath}
  where ``$q$'' stands for the category of quasi-isomorphisms.
\end{definition}

By Lemma~\ref{lem:H1_triv_high_twist}, the category $\ch\vb$ is the
increasing union of the categories~$\ch\vb_{n}$, $n \geq 0$. This
filtration allows us to restrict attention to the category $\vb_{0}$
as far as algebraic $K$\nbd-theory is concerned. Before we formally
prove this result, let us mention that twisting $\cY \mapsto \tw \cY k
\ell$ is an additively exact auto-equivalence that preserves
quasi-isomorphisms, cofibrations and pushouts, with the functor $\cZ
\mapsto \tw \cZ {-k} {-\ell})$ being an inverse. We also introduce new
notation for the twisting functor:
\begin{equation}
  \label{eq:theta}
  \theta_{k,\ell} \colon \vb \rTo \vb \ , \quad \cY \mapsto \tw \cY
  k \ell \ .
\end{equation}

\begin{lemma}
  For all $n \in \bN$ the inclusion $\vb_{n} \subseteq \vb_{n+1}$
  induces a homotopy equivalence $K\big(\vb_{n}\big) \simeq
  K\big(\vb_{n+1}\big)$.
\end{lemma}

\begin{proof}
  Let $k, \ell \in \bZ$. If $k + \ell \geq 0$ the
  functor~$\theta_{k,\ell}$ restricts to a functor
  \begin{displaymath}
    \vb_{n} \rTo \vb_{n}
  \end{displaymath}
  which we denote by the same symbol~$\theta_{k,\ell}$. In case
  $k+\ell \geq 1$ we may consider this as a functor
  $\theta_{k, \ell} \colon \vb_{n+1} \rTo \vb_{n}$.  From the exact
  sequence of twisting sheaves~\eqref{eq:cartesian} we conclude that
  there is a short exact sequence of functors
  \begin{displaymath}
    0 \rTo \id \rTo \theta_{1,0} \oplus \theta_{0,1} \rTo \theta_{1,1}
    \rTo 0
  \end{displaymath}
  (we have used $\theta_{0,0} \iso \id$ here). By the additivity
  theorem (\cite{MR802796}, Proposition~1.3.2~(4) and Theorem~1.4.2)
  that means that, on the level of $K$\nbd-groups, the identity map is
  the difference of the maps induced by
  $\theta_{1,0} \oplus \theta_{0,1}$ and $\theta_{1,1}$, both as
  endofunctors of $\vb_{n}$ and of $\vb_{n+1}$. This in turn implies
  that the difference of the induced maps of the functors
  $\theta_{1,0} \oplus \theta_{0,1}$ and $\theta_{1,1}$, considered as
  functors $\vb_{n+1} \rTo \vb_{n}$, is both left and right inverse to
  the map in $K$-theory induced by the inclusion functor
  $\vb_{n} \rTo^{\subseteq} \vb_{n+1}$.
\end{proof}

\begin{corollary}
  \label{cor:reduce_to_vb0}
  The inclusion $\vb_{0} \subseteq \vb$ induces a homotopy equivalence
  $\Omega |q \mathcal{S}_{\bullet} \ch \vb_{0}| \rTo^{\simeq}
  K(\pp)$.\qed
\end{corollary}

\section{Canonical sheaves and global sections}

We write $\chm$ for the category of bounded chain complexes of (right)
$R_{0}$\nbd-modules. Its full subcategory of bounded chain complexes
of finitely generated projective $R_{0}$\nbd-modules is
denoted~$\chp$. We consider the latter as a \textsc{Waldhausen}
category with weak equivalences the homotopy equivalences (or
equivalently, the quasi-isomorphisms), and cofibrations the injections
with levelwise projective cokernel. The associated $K$\nbd-theory
space is denoted by $K(R_{0}) = \Omega | h \sdot \chp |$; its homotopy
groups are the \textsc{Quillen} $K$\nbd-groups of the ring~$R_{0}$.

\medbreak

We define functors, the {\it canonical sheaves functors},
\begin{displaymath}
  \Psi_{k, \ell} \colon \chp \rTo \ch\vb \ , \quad C \mapsto C \tensor
  \OO k \ell
\end{displaymath}
where $C \tensor \OO k \ell$ denotes the bounded complex of sheaves
\begin{displaymath}
  C \tensor_{R_{0}} \Rnn k \rTo C \tensor_{R_{0}} \R \lTo C
  \tensor_{R_{0}} \Rpn {-\ell} \ .
\end{displaymath}
By Lemma~\ref{lem:projective_modules} the constituents of~$\OO k \ell$
are invertible bimodules over their respective ground rings. Hence
$C \tensor \OO k \ell$ consists of finitely generated projective
modules, and the functor $\Psi_{k,\ell}$ is an exact functor between
\textsc{Waldhausen} categories.

\begin{lemma}
  \label{lem:Psi_Gamma_adjoint}
  There are isomorphisms of \textsc{abel}ian groups
  \[\hom \big( \Psi_{0,0} (C),\, \cY \big) \iso \hom \big( C,\, \hh 0
  \cY \big)\]
  which are natural in~$C$ and in~$\cY$. In other words, the functor
  $\Psi_{0,0}$ is left adjoint to the global sections functor
  $\hh 0 {\nix}$.
\end{lemma}

\begin{proof}
  Since $\hh 0 \cY = \invlim \cY$, linear maps $C \rTo \hh 0 \cY$
  correspond bijectively to commutative diagrams of the type
  \begin{diagram}
    C & \rTo^{=} & C & \lTo^{=} & C \\ %
    \dTo && \dTo && \dTo \\ %
    Y^{-} & \rTo & Y^{0} & \lTo & Y^{+} %
  \end{diagram}
  with $\R_{0}$\nbd-linear vertical maps. By the usual extension of
  scalars construction, such diagrams correspond bijective to diagrams
  of the type
  \begin{diagram}
    C \tensor_{\R_{0}} \Rn & \rTo[l>=3em] & C \tensor_{\R_{0}} \R & %
    \lTo[l>=3em] & C \tensor_{\R_{0}} \Rp \\ %
    \dTo && \dTo && \dTo \\ %
    Y^{-} & \rTo & Y^{0} & \lTo & Y^{+} %
  \end{diagram}
  which are maps $\Psi_{0,0} (C) \rTo \cY$.
\end{proof}

\begin{lemma}
  \label{lem:H_of_Psi}
  For $C \in \chp$ we have natural isomorphisms
  \begin{displaymath}
    \Hh q {\Psi_{k, \ell} (C)} \iso C \tensor_{R_{0}} \Hh q {\OO k
      \ell} \ , \quad q = 0,1 \ ,
  \end{displaymath}
  where sheaf cohomology is computed in each chain level separately. 
  In particular, 
  \begin{enumerate}
  \item $\Hh 0 {\Psi_{0,0} (C)} \iso C$,
  \item $\Hh 0 {\Psi_{k,\ell} (C)} = 0$ if $k + \ell \leq -1$,
  \item $\Hh 1 {\Psi_{k,\ell} (C)} = 0$ if $k + \ell \geq -1$.
  \end{enumerate}
\end{lemma}

\begin{proof}
  For $\cY = \OO {k} {\ell}$ the canonical
  sequence~\eqref{eq:canonical} becomes an exact sequence of
  $R_{0}$-$R_{0}$-bimodule chain complexes
  \begin{multline*}
    0 \rTo \Hh 0 {\OO {k} {\ell}} \rTo \Rnn {k} \oplus \Rpn {-\ell}
    \\%
    \rTo \R \rTo \Hh 1 {\OO {k} {\ell}} \rTo 0 \ .
  \end{multline*}
  Tensoring this with $C$ over~$R_{0}$ yields another exact sequence
  as $C$ consists of projective $R_{0}$\nbd-modules; the result is in
  fact the canonical sequence~\eqref{eq:canonical} for the chain
  complex of vector bundles $\cY = \Psi_{k,\ell}(C)$, which thus has
  $q$th cohomology isomorphic to
  $C \tensor_{R_{0}} \Hh q {\OO {k} {\ell}}$. --- The assertions (1),
  (2) and~(3) follow from what we just proved and the calculations of
  Proposition~\ref{prop:H_of_On}: \goodbreak
  \begin{enumerate}
  \item $\hh 0 {\OO 0 0} \iso \R_{0}$;
  \item $\hh 0 {\OO k \ell} = 0$ if $k+\ell \leq -1$;
  \item $\hh 1 {\OO k \ell} = 0$ if $k+\ell \geq -1$.
  \end{enumerate}
\end{proof}

\begin{corollary}
  \label{cor:Psi_factors}
  Let $i,j \in \bZ$. If $i + j \geq -1$, the functor $\Psi_{i,j}$
  factors through the category~$\ch \vb_{0}$.
\end{corollary}

\begin{proof}
  We need to verify that $\Hh 1 {\tw {\Psi_{i,j}(C)} k \ell} = 0$ if
  $k+\ell \geq 0$, for any chain complex $C \in \chp$.
  By~\eqref{eq:twisting} we have an isomorphism
  $\tw {\Psi_{i,j}(C)} k \ell \iso \Psi_{i+k, j+\ell} (C)$ and thus,
  by the previous Lemma, the vanishing result $\Hh 1 {\tw {\Psi_{i,j}(C)} k \ell} = 0$.
\end{proof}

For $i+j \geq -1$ the functor $\Psi_{i,j}$ associates a bounded
complex of vector bundles in~$\vb_{0}$ with a bounded complex of
$R_{0}$\nbd-modules. We will need a functor going the other direction;
we use the global sections functor $\Gamma = \hh 0
{\nix}$. Explicitly,
\begin{displaymath}
  \Gamma \colon \ch \vb_{0} \rTo \chp \ , \quad \cY \mapsto \hh 0 \cY \ .
\end{displaymath}
This is well defined by virtue of
Theorem~\ref{thm:crucial_finiteness}: the complex $\hh 0 \cY$ consists
of finitely generated projective $R_{0}$\nbd-modules for
$\cY \in \ch \vb_{0}$. Note also that $\Gamma$ is an (additive) exact
functor, \ie, $\Gamma$ maps short exact sequences in~$\ch\vb_{0}$ to
short exact sequences in~$\chp$; this follows immediately from the
fact that twisting is an additive exact functor, and
from~\eqref{eq:not-so-long} as all first cohomology modules
vanish. Consequently, the functor~$\Gamma$ automatically preserves all
pushouts that exist in~$\vb_{0}$.

\begin{lemma}
  \label{lem:Gamma_exact}
  The functor $\Gamma \colon \ch\vb_{0} \rTo \chp$ is an exact functor
  of \textsc{Waldhausen} categories.
\end{lemma}

\begin{proof}
  The non-trivial points to check are the following:
  \begin{enumerate}
  \item If $f \colon \cY \rTo \cZ$ is a cofibration, then $\Gamma(f)$
    is a cofibration.
  \item If $f \colon \cY \rTo \cZ$ is a weak equivalence, then
    $\Gamma(f)$ is a weak equivalence.
  \end{enumerate}

  To prove~(1) recall that $\coker (f)$ is, by definition of
  cofibrations, a complex of vector bundles, by
  Lemma~\ref{lem:coker_is_vb_n} in fact a complex of vector bundles
  in~$\vb_{0}$. As $\Gamma$ is additively exact we obtain a short
  exact sequence
  \begin{displaymath}
    0 \rTo \Gamma (\cY) \rTo^{\Gamma(f)} \Gamma (\cZ) \rTo \Gamma
    (\coker f) \rTo 0
  \end{displaymath}
  of $R_{0}$\nbd-modules. But this sequence consists of chain
  complexes of projective modules, so $\Gamma(f)$ is a levelwise split
  injection with levelwise projective cokernel, that is, is a
  cofibration in~$\chp$.

  As for~(2), since we are working with vector bundles in~$\vb_{0}$
  the cohomology sequence~\eqref{eq:not-so-long} reduces to a short
  exact sequence
  \begin{displaymath}
    0 \rTo \Gamma(\cY) \rTo Y^{-} \oplus Y^{+}
    \rTo^{\upsilon^{-}-\upsilon^{+}} Y^{0} \rTo 0
  \end{displaymath}
  so that $\Gamma(\cY)$ is quasi-isomorphic to the ``homotopy fibre''
  $\mathbf{H}(\upsilon^{-}-\upsilon^{+})$ of the map
  $\upsilon^{-}-\upsilon^{+}$, \ie, to the totalisation of
  $Y^{-} \oplus Y^{+} \rTo^{\upsilon^{-}-\upsilon^{+}} Y^{0}$
  considered as a double complex. But this latter construction is
  known to be invariant under quasi-isomorphisms, so that a weak
  equivalence $f \colon \cY \rTo^{\sim} \cZ$ yields a square diagram
  \begin{diagram}
    \Gamma(\cY) & \rTo[l>=2.5em]^{\sim} & \mathbf{H}(\upsilon^{-}-\upsilon^{+}) \\
    \dTo<{\Gamma(f)} && \dTo<{\sim} \\
    \Gamma(\cZ) & \rTo^{\sim} & \mathbf{H}(\zeta^{-}-\zeta^{+})
  \end{diagram}
  whence $\Gamma(f)$ is a quasi-isomorphism as required.
\end{proof}

\section{Weak equivalences of complexes of sheaves}
\label{sec:weqs}

As usual in the context of \textsc{Waldhausen} $K$-theory we need
various notions of weak equivalences on the category concerned. As
defined before, a quasi-iso\-mor\-phism in $\ch \vb_0$ is map
$f \colon \cY \rTo \cZ$ such that its components~$f^{?}$, for
$? \in \{-,0,+\}$, are quasi-isomorphisms of chain complexes of
modules.

\begin{lemma}
\label{lem:characterise_h_0_equiv}
  Let $f \colon \cY \rTo \cZ$ be a map in $\ch \vb$, and let
  $k_{0},\ell_{0},n \in \bZ$ with $k_{0} + \ell_{0} = n$ be fixed
  integers. Each of the following statements implies the others:
  \begin{enumerate}
  \item The map $\Gamma \circ \theta_{k_{0},\ell_{0}} (f)$ is a
    quasi-isomorphism in~$\chm$.
  \item There exist $k,\ell \in \bZ$ with $k + \ell = n$ such that the
    map $\Gamma \circ \theta_{k,\ell} (f)$ is a quasi-isomorphism
    in~$\chm$.
  \item For all $k,\ell \in \bZ$ with $k + \ell = n$, the map
    $\Gamma \circ \theta_{k,\ell} (f)$ is a quasi-isomorphism
    in~$\chm$.
  \end{enumerate}
\end{lemma}

\begin{proof}
  First, we may replace the map $f$ by $\theta_{n,0}(f)$ and thus
  reduce to the case $n = 0$, using the
  isomorphisms~(\ref{eq:twisting}). So assume $n = 0$. For
  $k + \ell = 0$ there is a natural isomorphism
  $\Gamma \circ \theta_{k,\ell} (f) \iso \Gamma (f) \tensor_{R_{0}}
  R_{k}$,
  by Theorem~\ref{thm:cohomology_of_twist}. As $R_{k}$ is a projective
  left $R_{0}$\nbd-module by Lemma~\ref{lem:projective_modules}, and
  even an invertible $R_{0}$\nbd-bimodule by Lemma~\ref{lem:bimodules},
  we conclude that $\Gamma(f)$ is a quasi-isomorphism if and only if
  $\Gamma \circ \theta_{k,\ell}$ is. The result follows.
\end{proof}

\begin{definition}
  \label{def:h_0_1}
  A map $f \colon \cY \rTo \cZ$ in $\ch \vb_0$ is a {\it
    $q_{0}$\nbd-equivalence\/} if it satisfies one (and hence all) of
  the three equivalent conditions listed in
  Lemma~\ref{lem:characterise_h_0_equiv} for $n=0$. We call $f$ a {\it
    $q_{1}$\nbd-equivalence\/} if it satisfies one (and hence all) of
  the equivalent conditions listed in
  Lemma~\ref{lem:characterise_h_0_equiv} for $n=0$ and for $n=1$.
\end{definition}

Equipped with $q_{?}$\nbd-equivalences, $? = 0,1$, and the previous
notion of cofibrations the category $\ch \vb_{0}$ is a
\textsc{Waldhausen} category satisfying the saturation and extension
axioms. The gluing lemma follows from the fact that twisting preserves
pushouts, and from the fact that~$\Gamma$ preserves pushouts and
cofibrations by Lemma~\ref{lem:Gamma_exact}, as the category $\chp$ is
known to satisfy the gluing lemma. The usual mapping cylinder
construction provides a cylinder functor which satisfies the cylinder
axiom.

\section{Acyclic complexes of vector bundles}

We turn our attention to a characterisation of acyclic complexes of
vector bundles with trivial higher cohomology: If global sections of
``zeroth'' and ``first'' twists are trivial, the complex is
quasi-isomorphic to the zero complex. The precise result is formulated
as follows:

\begin{proposition}
  \label{prop:0_1__acyclic_is_trivial}
  Let $\cY \in \ch \vb_{0}$ be a complex of vector bundles such that
  the canonical map $\zeta \colon \cY \rTo 0$ is a
  $q_{1}$\nbd-equivalence, cf.~Definition~\ref{def:h_0_1}. Then
  $\zeta$ is a $q$\nbd-equivalence, that is, $\cY$ is quasi-isomorphic
  to the zero complex.
\end{proposition}

\begin{proof}
  The exact sequence of sheaves
  \begin{displaymath}
    0 \rTo \OO 0 0 \rTo \OO 1 0 \rTo \OO 1 0 / \OO 0 0 \rTo 0
  \end{displaymath}
  yields, by pointwise tensor product with the complex of vector
  bundles~$\cY$, an exact sequence of complexes of sheaves
  \begin{displaymath}
    0 \rTo \cY \rTo^{\alpha} \cY(1,0) \rTo \cZ \rTo 0
  \end{displaymath}
  with $\cZ = \coker(\alpha)$ satisfying $Z^{0} = 0$ and $Z^{+} = 0$.
  Explicitly, we obtain the diagram of chain complexes of
  $R_{0}$\nbd-modules with exact columns depicted in Fig.~\ref{fig:3x3}.
  \begin{figure}[ht]
    \centering
    \begin{diagram}[small]
      0 && 0 && 0 \\%
      \dTo && \dTo && \dTo \\
      Y^{-} \tensor_{\Rn} \Rnn 0 & \rTo[l>=4em] & Y^{0} \tensor_{\R}
      \R & \lTo[l>=4em] & Y^{+} \tensor_{\Rp} \Rpn 0 \\%
      \dTo>{\alpha^{-}} && \dTo>{\alpha^{0} = \id} && \dTo>{\alpha^{+}
        = \id} \\%
      Y^{-} \tensor_{\Rn} \Rnn 1 & \rTo & Y^{0} \tensor_{\R} \R & \lTo
      & Y^{+} \tensor_{\Rp} \Rpn 0 \\%
      \dTo && \dTo && \dTo \\%
      Y^{-} \tensor_{\Rn} R_{1} & \rTo & 0 & \lTo & 0 \\%
      \dTo && \dTo && \dTo \\
      0 && 0 && 0 %
    \end{diagram}
    \caption{Diagram used in proof of
      Proposition~\ref{prop:0_1__acyclic_is_trivial}}
    \label{fig:3x3}
  \end{figure}
  The $\Rn$\nbd-module structure of $R_{1}$ in the third non-trivial
  row is determined by saying that all elements of strictly negative
  degree act on~$R_{1}$ as multiplication by~$0$. --- The first
  horizontal slice is $\tw {\cY} 0 0 \iso \cY$, while the second is
  the vector bundle $\theta_{1,0} (\cY) = \tw {\cY} 1 0$. The first
  slice has levelwise trivial $\invlim{}^{\!1}$; this is true by
  definition of~$\vb_{0}$ and by the equality
  $\invlim{}^{\!1} = \hh 1 {\nix}$.  As a consequence, applying the
  functor $\invlim$ results in a short exact sequence of
  $R_{0}$\nbd-modules
  \begin{displaymath}
    0 \rTo \Gamma (\cY) \rTo \Gamma \circ \theta_{1,0} (\cY) \rTo
    Y^{-} \tensor_{\Rn} R_{1} \rTo 0
  \end{displaymath}
  with contractible first and second terms, by our hypothesis
  on~$\cY$. So the complex
  $Y^{-} \tensor_{\Rn} R_{1} = \coker (\alpha^{-})$ has trivial
  homology. But $\alpha^{-}$ is a map of bounded chain complexes of
  projective $\Rn$\nbd-modules, whence $\alpha^{-}$ is actually a
  homotopy equivalence. By tensoring over~$\Rn$ with~$\Rnn p$ from
  the right, and using the $\Rn$\nbd-bimodule isomorphism
  $\Rnn {p'} \tensor_{\Rn} \Rnn{p} \iso \Rnn {p'+p}$, we conclude that the
  map
  \begin{displaymath}
    \alpha^{-}_{p} \colon Y^{-} \tensor_{\Rn} \Rnn {p} \rTo Y^{-}
    \tensor_{\Rn} \Rnn {p+1}
  \end{displaymath}
  is a homotopy equivalence for all $p \geq 0$. But then the map
  $\mu^{-}$ from $Y^{-} \iso Y^{-} \tensor_{\Rn} \Rnn 0$ to the
  colimit~$W^{-}$ of the sequence
  \begin{displaymath}
    Y^{-} \tensor_{\Rn} \Rnn 0 \rTo^{\alpha^{-}_{0}}  Y^{-} \tensor_{\Rn}
    \Rnn {1} \rTo^{\alpha^{-}_{1}}  Y^{-} \tensor_{\Rn} \Rnn {2}
    \rTo^{\alpha^{-}_{2}} \ldots
  \end{displaymath}
  is a quasi-isomorphism since homology commutes with filtered
  colimits. As this colimit is precisely
  $W^{-} = Y^{-} \tensor_{\Rn} \bigcup_{p \geq 0} \Rnn p = Y^{-}
  \tensor_{\Rn} \R$,
  and as $\cY$ is a complex of vector bundles, the composite map
  \begin{displaymath}
    Y^{-} \rTo^{\mu^{-}}_{\sim} W^{-} = Y^{-} \tensor_{\Rn} \R
    \rTo_{\iso}^{\upsilon^{-}_{\sharp}} Y^{0} 
  \end{displaymath}
  is a quasi-isomorphism. That is, the structure map
  $Y^{-} \rTo^{\upsilon^{-}} Y^{0}$ of~$\cY$ is a
  quasi-isomorphism. --- A similar argument, employing 
  $\cY(0,1)$ in place of~$\cY(1,0)$, shows that the structure map
  $Y^{0} \lTo Y^{+}$ is a quasi-isomorphism as well.

  We have shown that $\cY$ is quasi-isomorphic, as a diagram of chain
  complexes of projective $R_{0}$-modules, to the constant diagram
  \begin{displaymath}
    \cC_{Y} = \Big( Y^{0} \rTo^{\id} Y^{0} \lTo^{\id} Y^{0} \Big) \ ,
  \end{displaymath}
  which has levelwise trivial $\invlim{}^{\!1}$. It follows that
  \begin{displaymath}
    Y^{0} \underset {\text{(i)}} \iso \invlim \cC_{Y} \underset
    {\text{(ii)}} \simeq \invlim \cY = \Gamma (\cY) \simeq 0 \ ,
  \end{displaymath}
  that is, $Y^{0}$ is acyclic. (The isomorphism~(i) is the result of
  direct computation, while for the quasi-isomorphism~(ii) we can use
  an argument similar to the one used in Lemma~\ref{lem:Gamma_exact},
  part~(2) of the proof.) As $Y^{-}$ and~$Y^{+}$ are quasi-isomorphic
  to~$Y^{0}$, as shown above, this proves $\cY \simeq 0$ as required.
\end{proof}

\section{The splitting result}

\begin{theorem}
  Suppose that $R = \bigoplus_{k \in \bZ} R_{k}$ is a strongly
  $\bZ$\nbd-graded ring. There is a homotopy equivalence of
  $K$\nbd-theory spaces
  \begin{displaymath}
    K(R_{0}) \times
    K(R_{0}) \rTo K(\pp)
  \end{displaymath}
  induced by the functor
  \begin{align*}
    \Psi_{-1,0} + \Psi_{0,0} \colon \chp \times \chp & \rTo \ch \vb \\
    (C,D) & \rlap{\ \(\mapsto\)} \hphantom{\rTo} \,\Psi_{-1,0}(C)
            \oplus \Psi_{0,0} (D) \ .
  \end{align*}
\end{theorem}

\begin{proof}
  By Corollary~\ref{cor:Psi_factors} the functor
  $\Psi_{0,0} + \Psi_{-1,0}$ factors through the category
  $\ch \vb_{0}$ which, in view of Corollary~\ref{cor:reduce_to_vb0},
  has the same $K$\nbd-theory (with respect to $q$\nbd-equivalences)
  as the category~$\ch \vb$. It is thus enough to show that the
  functor
  \begin{align*}
    \Psi_{-1,0} + \Psi_{0,0} \colon \chp \times \chp & \rTo \ch \vb_{0}\\%
    (C,D) & \rlap{\ \(\mapsto\)} \hphantom{\rTo} \Psi_{-1,0}(C) \oplus
            \Psi_{0,0} (D)
  \end{align*}
  induces a homotopy equivalence on $K$\nbd-theory spaces (with
  respect to quasi-iso\-mor\-phisms as weak equivalence in both source
  and target).
 
  \begin{claim}
    Let $q_{0}$ denote the category of $q_{0}$\nbd-equivalences
    in~$\vb_{0}$, see Definition~\ref{def:h_0_1}.  There is a
    fibration sequence
    \begin{equation}
      \label{eq:fib0}
      |q \sdot \ch \vb_{0}^{q_{0}}| \rTo |q \sdot \ch \vb_{0}|
      \rTo^{\Gamma} |h \sdot \chp| \qquad
    \end{equation}
    which has a section up to homotopy induced by the
    functor~$\Psi_{0,0}$.
  \end{claim}

  Indeed, every quasi-isomorphism in~$\ch \vb_{0}$ is a
  $q_{0}$\nbd-equi\-valence. The machinery of \textsc{Waldhausen}
  $K$\nbd-theory \cite[Theorem~1.6.4]{MR802796} thus provides us with
  a fibration sequence
  \begin{displaymath}
    |q \sdot \ch \vb_{0}^{q_{0}}| \rTo |q \sdot \ch \vb_{0}| \rTo
    |q_{0} \sdot \ch \vb_{0}| \ .
  \end{displaymath}
  The base is homotopy equivalent to $|h \sdot \chp|$, \textit{via}
  the functors~$\Gamma$ and~$\Psi_{0,0}$; in fact,
  $\Gamma \circ \Psi_{0,0} \iso \id$ by Lemma~\ref{lem:H_of_Psi}, 
  and the adjunction counit
  $\epsilon \colon \Psi_{0,0} \circ \Gamma \rTo \id$ (associated with
  the adjunction of Lemma~\ref{lem:Psi_Gamma_adjoint}) is a
  $q_{0}$\nbd-equivalence as application of~$\Gamma$ yields, up to
  canonical isomorphism, the identity transformation
  $\id = \Gamma(\epsilon) \colon \Gamma \circ \Psi_{0,0} \circ \Gamma
  \rTo \Gamma$.
  From $\Gamma \circ \Psi_{0,0} \iso \id$ we also conclude that
  $\Psi_{0,0}$ induces a section up to homotopy of the fibration
  sequence~\eqref{eq:fib0}. This proves the claim.

  \medbreak

  Since the fibration sequence~\eqref{eq:fib0} has a section up to
  homotopy, the associated long exact sequence of homotopy groups
  (which are \textsc{abel}ian groups in this case) decomposes into
  split short exact sequences
  \begin{multline*}
    0 \rTo \pi_k |q \sdot \ch \vb_{0}^{q_{0}}| \rTo \pi_k |q \sdot \ch
    \vb_{0}| \\ %
    \rTo^{\Gamma} \pi_k |h \sdot \chp| \rTo 0
  \end{multline*}
  so that there are group isomorphisms
  \begin{displaymath}
    \pi_k |q \sdot \ch \vb_{0}^{q_{0}}| \,\oplus\, \pi_k |h \sdot
    \chp| \rTo^{\lambda} \pi_k |q \sdot \ch \vb_{0}|
  \end{displaymath}
  with $\lambda$ induced by the functor
  $(\cY, \cZ) \mapsto \cY \oplus \Psi_{0,0} (\cZ)$. It follows that
  this functor yields a homotopy equivalence
  \begin{equation}
    \label{eq:he1}
    |q \sdot \ch \vb_{0}^{q_{0}}| \times |h \sdot \chp|
    \rTo[l>=5em]^{(\mathrm{incl}\ \Psi_{0,0})}_{\simeq} |q \sdot \ch
    \vb_{0}| \ .
  \end{equation}

  \smallbreak

  \noindent We proceed by analysing the fibre term of the fibration
  sequence~\eqref{eq:fib0}.

  \begin{claim}
    There is another fibration sequence, similar to the one above,
    \begin{equation}
      \label{eq:fib1}
      |q \sdot \ch \vb_{0}^{q}| \rTo |q \sdot \ch \vb_{0}^{q_{0}}|
      \rTo[l>=4em]^{\Gamma \circ \theta_{1,0}} |h \sdot \chp|
    \end{equation}
    which has a section up to homotopy induced by the
    functor~$\Psi_{-1,0}$.
  \end{claim}
  
  Indeed, let $q_{1}$ denote the category of $q_{1}$\nbd-equivalences
  in~$\vb_{0}^{q_{0}}$, see Definition~\ref{def:h_0_1}. The general
  $K$\nbd-theory machinery provides a fibration sequence
  \begin{displaymath}
    |q \sdot \ch \vb_{0}^{q_{1}}| \rTo |q \sdot \ch
    \vb_{0}^{q_{0}}| \rTo |q_{1} \sdot \ch \vb_{0}^{q_{0}}| \ .
  \end{displaymath}
  Again, the base is homotopy equivalent to~$|h \sdot \chp|$
  \textit{via} the functors $\Gamma \circ \theta_{1,0}$
  and~$\Psi_{-1,0}$; note that $\Gamma \circ \Psi_{-1,0} = 0$ by
  Lemma~\ref{lem:H_of_Psi} 
  so that
  $\Psi_{-1,0}$ takes indeed values in~$\vb_{0}^{q_{0}}$. We have
  $\big( \Gamma \circ \theta_{1,0} \big) \circ \Psi_{-1,0} = \Gamma
  \circ \Psi_{0,0} \iso \id$, and the adjunction counit
  \begin{displaymath}
    \Psi_{-1,0} \circ \big(\Gamma \circ \theta_{1,0}
    \big) \rTo \id
  \end{displaymath}
  is transformed, by application of~$\Gamma \circ \theta_{1,0}$, into
  the identity map. So there results a fibration sequence
  \begin{equation}
    \label{eq:1}
    |q \sdot \ch \vb_{0}^{q_{1}}| \rTo |q \sdot \ch
    \vb_{0}^{q_{0}}| \rTo[l>=4em]^{\Gamma \circ \theta_{1,0}} |h \sdot
    \chp| \ ,
  \end{equation}
  and the isomorphism
  $\big( \Gamma \circ \theta_{1,0} \big) \circ \Psi_{-1,0} \iso \id$
  shows that $\Psi_{-1,0}$ provides a section up to homotopy. The
  category $\ch \vb_{0}^{q_{1}}$ in the fibre term consists of those
  complexes of vector bundles $\cY \in \vb_{0}$ such that $\cY \rTo 0$
  is a $q_{1}$\nbd-equivalence, that is, by
  Proposition~\ref{prop:0_1__acyclic_is_trivial}, such that
  $\cY \simeq 0$. Thus we have an equality of categories
  $\ch \vb_{0}^{q_{1}} = \ch \vb_{0}^{q}$, and the fibration
  sequence~\eqref{eq:1} is actually the desired
  sequence~\eqref{eq:fib1}. Its fibre term is contractible. This means
  that $\Gamma \circ \theta_{1,0}$ induces a homotopy equivalence;
  consequently, the section up to homotopy, induced by~$\Psi_{-1,0}$,
  is a homotopy equivalence as well.

  Combining this homotopy equivalence with~(\ref{eq:he1}) yields a
  homotopy equivalence
  \begin{displaymath}
    |h \sdot \chp| \times |h \sdot \chp| \rTo[l>=6em]^{(\Psi_{-1,0} \
      \Psi_{0,0})}_{\simeq} |q \sdot \ch \vb_{0}| \ .
  \end{displaymath}
  \goodbreak
  Looping once and combining with the homotopy equivalence
  \begin{displaymath}
    \Omega |h \sdot \ch \vb_{0}| \rTo^{\simeq} \Omega |h \sdot \ch
    \vb| = K(\pp)
  \end{displaymath}
  from Corollary~\ref{cor:reduce_to_vb0} gives the advertised
  splitting result.
\end{proof}

\vspace{2 em}

\raggedright


\end{document}